\documentclass[12pt]{amsart}
\usepackage{geometry}
\usepackage{graphicx}
\usepackage{amscd}
\usepackage{amssymb}
\usepackage{amsmath}
\usepackage[latin1]{inputenc}
\usepackage{amsfonts}
\usepackage[all]{xy}

\theoremstyle{plain}                                       %
\newtheorem{thm}{\quad Theorem}                            %
\newtheorem{cor}[thm]{\quad Corollary}                     %
\newtheorem{prop}[thm]{\quad Proposition}                  %
\theoremstyle{definition}                                  %
\newtheorem{defi}[thm]{\quad Definition}                   %
\newtheorem{rmk}[thm]{\quad Remark}                        %
\newtheorem{ejem}[thm]{\quad Example}                      %

\newcommand{\R}{{\Bbb R}}
\newcommand{\N}{{\Bbb N}}
\newcommand{\K}{{\Bbb K}}
\newcommand{\C}{{\Bbb C}}

\newcommand{\al}{{\alpha}}

\newcommand{\Ri}{{\mathcal{R}}}

\begin{document}

\vspace{1cm}

\title{Characteristic Curves and the exponentiation in the Riordan Lie group: A connection through examples.}

\author{ Pedro J. Chocano$^\flat$,  Ana Luz\'{o}n*, Manuel A. Mor\'{o}n$^\natural$, L. Felipe Prieto Mart\'{i}nez$^\dag$ }
\address{$^\flat$ Departamento de Matem\'{a}tica Aplicada, Ciencia e Ingenier\'{i}a de los Materiales y Tecnolog\'{i}a Electr\'{o}nica, ESCET Universidad Rey Juan Carlos, 28933 M\'{o}stoles (Madrid), Spain}
\email{pedro.chocano@urjc.es}
\address{*Departamento de Matem\'{a}tica Aplicada. Universidad Polit\'{e}cnica de Madrid (Spain).}
\email{anamaria.luzon@upm.es}

\address{ $\natural$ Departamento de Algebra, Geometr\'{i}a y Topolog\'{i}a.
Universidad Complutense de Madrid  and Instituto de Matem\'{a}tica
Interdisciplinar (IMI)(Spain).} \email{mamoron@mat.ucm.es}

\address{ $^\dag$  Departamento de Matem\'{a}tica Aplicada, ETS Arquitectura, Universidad Polit\'{e}cnica de
Madrid (Madrid). }
\email{luisfelipe.prieto@upm.es}

\maketitle


\begin{abstract}
We point out how to use the classical characteristic method, that is used to solve quasilinear PDE's, to obtain the matrix exponential of some lower triangle infinite matrices. We use the Lie Frechet structure of the Riordan group described in \cite{Lie}. After that we describe some linear dynamical systems in $\K[[x]]$ with a concrete involution being a symmetry or a time-reversal symmetry for them. We take this opportunity to assign some dynamical properties to the Pascal Triangle. 
\end{abstract}

\maketitle

\section{Introduction}

Several phenomena may be modelled using systems of differential equations and a key notion to get their solutions is the exponential of (finite) square matrices. Therefore, the problem of computing these exponentials is relevant. Moreover, some crucial developments of Linear Algebra have been motivated by the study of the matrix exponential. We recommend \cite{mol-van} for a survey on this topic.

Furthermore, the matrix exponential lies in the core of Lie Theory and connects the natural Lie algebra of all $n\times n $ matrices with the general linear group $GL(n, \K)$. We think that \cite{Rossmann} is a very good introductory text for this point of view.

By allowing concepts as manifolds modelled in infinite dimensional spaces (e.g., Banach spaces, Frechet spaces etc.), the theory of Lie groups has been extended to the infinite dimensional framework \cite{Milnor}. See \cite{H-N} for some coherence relationships between infinite dimensional Lie group structures and pro-Lie group structures when both are shared by a group. This is the case of the Riordan group. The previous facts motivated us  to introduce and develop in \cite{Lie} a structure of infinite dimensional Lie group on the Riordan group. 

In the following section we recall some basic facts about the Riordan group and its Lie group and pro-Lie group structures described in \cite{Lie}. Notice that although this group appeared under this name more or less recently, the group structure and many of its elements are latent in many developments of classical mathematics (special sequences of numbers and polynomials, Umbral Calculus and much more). It is clear that,  historically,  the first (and surely the best) known Riordan matrix is Pascal's Triangle. Besides this, now, there is a lot of historical names of mathematicians related to some Riordan matrices.

The aim of this note is to describe the matrix exponential, for some infinite lower triangular matrices, from the solution of certain partial differential equations, which we find using the method of characteristics (see \cite{john}). We do all of this within the framework of the Riordan Lie group and the corresponding Lie algebra. Then, motivated by general symmetry properties of dynamical systems (\cite{Lam-Rob} and \cite{ofarrel2}), we describe some of those systems in $\K[[x]]$ that come from the Riordan group with a very special Riordan involution as a symmetry or time-reversal symmetry. We consider a sequence of linear ordinary differential equations in Euclidean spaces as problems approaching certain partial differential equations, using for that the pro-structures of both: the Riordan group and the corresponding Lie algebra. We propose a geometric analysis of these problems and we enumerate symmetric properties of a special sequence related to the Pascal Triangle.

Apart from the general description of the Riordan group  contained in Section $2$, we also need, along the paper, some specific results from \cite{Lie}. This is the reason why in Section $3$ and $4$ we recall some particularly related results from \cite{Lie} to make this note as self-contained as possible.

Since our first approach to the Riordan group, in  \cite{teo}, is different from the usual one in the literature, we have also different notation. This is the reason why we recommend  our previous works \cite{teo}, \cite{2ways}, \cite{finitas} and, of course, \cite{Lie} (ordered chronologically) for information about basic results and notation used herein. We still maintain the notation $\K$ for a field although in this paper we are only considering $\K$ as the real numbers. This is because many of the results and ideas can be translated, at least, to the case $\K=\C$, the field of complex numbers.

\section{Basic facts on the differentiable structure of the Riordan group} The results of this section can be found in \cite{teo},  \cite{2ways}, \cite{finitas} and  \cite{Lie}.
\subsection{Riordan matrices and the Riordan group}
The definition of {\it Riordan matrix} and the related concept  of {\it Riordan group} appeared in the foundational paper \cite{Sha91} due to Shapiro, Getu, Woan, and Woodson. The original definition of a Riordan matrix given in \cite{Sha91} is more restrictive than that used currently in the literature, which is precisely the one that we are going to use herein.

\textit{A Riordan matrix is an infinite matrix $D=(d_{i,j})_{i,j\in\N}$   whose columns are the coefficients of successive terms of a geometric progression in $\K[[x]]$ where the initial term is a formal power series of order $0$ and the common ratio is a formal power series of order $1$ (so $D$ needs to be lower triangular and its diagonal needs to be a geometric progression in $\mathbb{K}$). }
 
We represent a Riordan matrix $D$ by  $T(f\mid
g)$, where $f(x)=\sum_{k=0}^{\infty}f_k x^k$ and $g(x)=\sum_{k=0}^{\infty}g_k x^k$  are formal power series in $\K[[x]]$ with $f(0)\neq0$ and $g(0)\neq0$, so that
$\displaystyle{d_{i,j}=[x^i]\frac{x^jf(x)}{g^{j+1}(x)}}$.
Consequently, the first term of the geometric progression is $\displaystyle{\frac{f(x)}{g(x)}}$
and the common ratio is  $\displaystyle{\frac{x}{g(x)}}$. In this terms, Pascal's triangle is $T(1\mid 1-x)$.

The above definition can be reinterpreted saying that the generating function of the $j$-th column (starting at $j=0$) of $D$ is the formal power series $\frac{x^jf(x)}{g^{j+1}(x)}$, which makes sense because $g(0)\neq0$. Hence, $D$ is a lower triangular matrix and it is invertible because $f(0)\neq0$.

In \cite{Sha91} it was stated one of the main results about Riordan matrices. Currently many authors call it the
{\it Fundamental Theorem for Riordan matrices (FTRM)}. \textit{Let $D=T(f\mid g)$ be a Riordan matrix and let $\gamma(x)=\sum_{k=0}^{\infty}
\gamma_k x^k$ be a power series in $\K[[x]]$. Consider the column vector ${\bf c}=(\gamma_0, \gamma_1, \gamma_2, \cdots)^T$. Then, the generating function of the matrix product $D{\bf c}$ is $\frac{f(x)}{g(x)}\gamma(\frac{x}{g(x)})$.  }
This fact is represented by $T(f\mid g)(\gamma)=\frac{f(x)}{g(x)}\gamma(\frac{x}{g(x)})$. A proof of this result, using a special ultrametric space $(\K[[x]], d)$ can be found  in \cite[Proposition 19]{teo}.

The Riordan group (i.e., the set of all Riordan matrices with the usual product of matrices), denoted by $\Ri(\K)$  or shortly $\Ri$, is a subgroup of the group of invertible infinite lower triangular matrices with the usual product of matrices as the
operation. The product is given by 
\[T(f\mid g)T(l\mid m)=T\left(fl\left(\frac{x}{g}\right)\big|
gm\left(\frac{x}{g}\right)\right),\] where
$\displaystyle{fl\left(\frac{x}{g}\right)\equiv f(x)\cdot
l\left(\frac{x}{g(x)}\right)}$ and analogously for the second term, and the inverse is given by 
\[(T(f\mid g))^{-1}\equiv T^{-1}(f\mid g)=T\left(\frac{1}{f(\frac{x}{A})}\Big| A \right),\]
where
$\displaystyle{\left(\frac{x}{A}\right)\circ\left(\frac{x}{g}\right)=\left(\frac{x}{g}\right)\circ\left(\frac{x}{A}\right)=x}$. See \cite[Proposition 20]{teo} for more details.

The sequence of the coefficients of the previous formal power series, denoted by $A$, is the so-called 
$A$-sequence of $T(f\mid g)$. Obviously, the A-sequence of $T(f\mid g)$ depends only on the power series $g$. Moreover, if $A=\sum_{k\geq0}a_{k}x^{k}$, then
\[d_{i,j}=\sum_{k=0}^{i-j}a_kd_{i-1,j-1+k} \qquad i,j\geq1.\]

\subsection{The Lie and pro-Lie group structures on the Riordan group}
Suppose that $\K$ is the field of real or complex numbers, denoted
by $\R$ and $\C$ respectively. Let us consider a natural way to give a completely metrizable topology  in $\K[[x]]$, by means of the identification $\K^{\N}\equiv\K[[x]]$ obtained by passing from sequences to ordinary generating functions and vice versa.

The topology considered  in $\K^{\N}$ is always the product topology for the usual topology in $\K$. Therefore, we convert $\K[[x]]$ into a Frechet space, that is, a completely metrizable locally convex linear topological space.

This is the starting point to describe a natural Frechet Lie group structure on the Riordan group. Beside this, the Riordan group can be described as the inverse limit of an inverse sequence of groups of finite matrices obtaining a pro-Lie group structure on the Riordan group.

It is well known that any Riordan matrix is completely determined by its first column and
its $A$-sequence. In this way, any Riordan matrix $D=(d_{i,j})_{i,j\in\N}$ is defined
by a sequence $\textbf{u}=(u_k)_{k\in\N}$ with $u_0\neq0$,
$u_1\neq0$, and $u_{2k}=x_k$, $u_{2k+1}=a_k$, being $x_k=d_{k,0}$, $A(x)=\sum_{n\geq0}a_nx^n$ and $d_{i,j}=\sum_{k=0}^{i-j}a_kd_{i-1,j-1+k}$ for $j\geq1$. We denote the matrix $D$ described above by
$\varphi_{\infty}(\textbf{u})$.

Let us consider $\K$ with the usual Euclidean topology, the product topology in $\K^{\N}$ and the basic open set
\[
\mathcal{U}_{\infty}=\left\{\textbf{u}=(u_k)_{k\in\N}\in\K^{\N} \ |
\ \ u_0\neq0, \ u_1\neq0\right\}
\]
in $\K^{\N}$. Set
\[
\begin{matrix}
\varphi_{\infty}& : & \mathcal{U}_{\infty} &  \longrightarrow & \Ri(\K)\\
&& \textbf{u}& \longmapsto & \varphi_{\infty}(\textbf{u}).
\end{matrix}
\]
Obviously, $\varphi_{\infty}$ is a bijective function. So, we consider the unique topology on $\Ri(\K)$, that makes
 $\varphi_{\infty}$ a homeomorphism. Note that the topological space $\Ri(\K)$, the locally convex vector space $\K^{\N}$ and the map $\varphi_{\infty} : \mathcal{U}_{\infty}\rightarrow  \Ri(\K)$ fit all conditions to get:
 \begin{thm}
$(\Ri(\K), (\mathcal{U}_{\infty}, \varphi_{\infty}) )$ is a smooth
manifold modelled on the locally convex vector space $\K^{\N}$. Moreover, $\Ri(\K)$ with this smooth structure is a Lie group.
\end{thm}

One of the main tools that we have used to get results on the Riordan group is to consider it as the inverse limit of an inverse sequence of groups of finite matrices. A natural way to do this is as follows. For every $n\in\N$, consider the general linear group $GL(n+1,\K)$
formed by all $(n+1)\times(n+1)$ invertible matrices with coefficients in $\K$. Since every Riordan matrix is lower triangular, we can define a natural homomorphism $\Pi_n: \mathcal{R}\rightarrow GL(n+1,\K)$ given by
\[
\Pi_n((d_{i,j})_{i,j\in\N})=(d_{i,j})_{i,j=0,\cdots,n}.
\]
For obvious reasons, we will refer to this homomorphism
as the projection of the corresponding Riordan matrix.

To describe the Riordan group as an inverse limit of an inverse
sequence of groups of finite matrices we use the results of \cite{finitas}. We first consider the
subgroup of $GL(n+1,\K)$ defined by
$\mathcal{R}_n=\Pi_n(\mathcal{R})$.

\begin{defi}
Let $D=(d_{i,j})_{i,j=0,\cdots,n+1}\in\Ri_{n+1}$. We define
$P_n:\Ri_{n+1}\rightarrow\Ri_n$ by
\[
 P_{n}((d_{i,j})_{i,j=0,1,\cdots,n+1})=(d_{i,j})_{i,j=0,\cdots,n}.
\]
\end{defi}
$P_n(D)$ is obtained from $D$ by deleting its last row and its
last column. $P_n$ is a group homomorphism for every $n$ because
the matrices are lower triangular. Moreover, the diagram below is
commutative
\begin{displaymath}
\xymatrix{
&\Ri \ar[d]_{\Pi_n} \ar[rd]^{\Pi_{n+1}}& \\
&\Ri_{n}  & \ar[l]^{P_n} \Ri_{n+1} .}
\end{displaymath}
From this, we get
\begin{thm} The  Riordan group $\Ri$ is isomorphic to
$\underleftarrow{\lim}\{(\mathcal{R}_n)_{n\in\N},(P_n)_{n\in\N}\}$. Consequently, $\Ri$ is a pro-Lie group.
\end{thm}
See \cite{finitas} for more details and notation.
\subsection{The Lie Algebra of the Riordan group}

Again, using the results of \cite{Lie} we obtain a full and faithful representation of the Lie algebra
$\mathcal{L}(\Ri(\K))$ of the Lie group $\Ri(\K)$. We have that $\displaystyle{L=(\ell_{i,j})_{i,j\in\N}}\in \mathcal{L}(\Ri(\K))$
if and only if $L$ is lower triangular and there are two sequences
$(\chi_i)_{i\in\N}$ and $(\al_i)_{i\in\N}$ such that

{\small
\[
L=\left(
                                \begin{array}{ccccccc}
                                  \chi_0 &  &  &  &  &  &  \\
                                  \chi_1 & \chi_0+\al_0 &  &  &  &  &  \\
                                  \chi_2 & \chi_1+\al_1 & \chi_0+2\al_0 &  &  &  &  \\
                                   \vdots & \vdots & \vdots & \ddots &  &  \\
                                  \chi_{n-1} & \chi_{n-2}+\al_{n-2} & \chi_{n-3}+2\al_{n-3}&  \cdots & \chi_0+(n-1)\al_0  &  \\
                                  \chi_n & \chi_{n-1}+\al_{n-1} & \chi_{n-2}+2\al_{n-2} & \cdots & \chi_1+(n-1)\al_1 &  \chi_0+n\al_0\\
                                    \vdots & \vdots & \vdots & \cdots& \vdots&\vdots&\ddots  \\
                                \end{array}
                              \right)
\]} with the usual sum of matrices and the usual product by scalars in
$\K$. The Lie bracket is
\[
[L_1,L_2]=L_1L_2-L_2L_1
.\]
We denote the matrix $L$ above as $L(\chi,\alpha)$, where $\displaystyle{\chi(x)=\sum_{n\geq0}\chi_nx^n, \
\al(x)=\sum_{n\geq0}\al_nx^n}$, and 

\[
\mathcal{L}(\Ri(\K))=\left\{L(\chi,\al) \  |  \ \chi,
\al\in\K[[x]]\right\}.
\]

\begin{prop}
Any $L=L(\chi, \alpha)\in\mathcal{L}(\Ri(\K))$ induces a linear continuous map,
denoted again by $L:\K^{\N}\rightarrow\K^{\N}$, given by $L(h)=
\chi(x)h(x)+x\al(x)h'(x)$ where 
$h(x)=\sum_{n\geq0}h_nx^n$.
\end{prop}
The continuous linear map $L(\chi,\al)$ can be viewed as a
$C^{\infty}$ vector field in the Frechet space $\K^{\N}$ under the
canonical identification $T_{h}\K^{\N}=\K^{\N}$ in the tangent space
at any $h$. From this point of view we have the following
proposition.
\begin{prop}\label{prop:initialValue}
The initial value problem
\[
\left\{
  \begin{array}{ll}
    \gamma'(t)=L(\chi,\al)(\gamma(t)) \\
    \gamma(0)=h
  \end{array}
\right.
\]
in $\K^{\N}$ has a unique solution given by
\[
\gamma(t)=e^{tL(\chi,\al)}(h).
\]
Consequently, there is a one-parameter group of Riordan matrices
$T(f(x,t)\mid g(x,t))$ such that
\[
\gamma(t)=T(f(x,t)\mid g(x,t))(h) =\frac{f(x,t)}{g(x,t)}h\left(\frac{x}{g(x,t)}\right).
\] In particular, $\{T(f(x,t)\mid g(x,t))\}_{t\in\R} $ is an abelian subgroup of the Riordan group, because it defines a continuous dynamical system (or flow) in $\K[[x]]$.
\end{prop}
Due to the special patterns followed by the matrices in the Lie Algebra and those in the Riordan group, we named the corresponding section in \cite{Lie} as {\it Arithmetic vector fields and geometric flows.}
\subsection{The partial differential equation induced by an element in the Lie Algebra.}
For any $t\in \R$ and  $L\in
\mathcal{L}(\Ri(\K))$ we
have a well-defined Riordan matrix $e^{tL}$. Note that for any $t\in \R$, $\gamma(t)\in \K[[x]]$, where $\gamma$ satisfies the conditions of Proposition \ref{prop:initialValue}. With this in mind, another way to interpret
the above proposition is as follows.
\begin{cor}
Let $\chi,\al\in\K[[x]]$. Then the unique solution of the initial value
problem
\[
\left\{
  \begin{array}{ll}
    \frac{\partial u}{\partial t}=\chi(x)u(x,t)+x\al(x)\frac{\partial u}{\partial x} \\
    u(x,0)=h(x)
  \end{array}
\right.
\]
in $\K[[x,t]]$ is given by
\[
u(x,t)=e^{tL(\chi,\al)}(h(x))=\frac{f(x,t)}{g(x,t)}h\left(\frac{x}{g(x,t)}\right).
\]
\end{cor}

\section{From the matrix exponential $e^{tL}$ to the solutions of the PDE $\frac{\partial u}{\partial t}=\chi(x)u(x,t)+x\al(x)\frac{\partial u}{\partial x}$ and back}
Once recalled some basic results from \cite{Lie}, we propose the following \textit{strategy}.

Given any element $L=L(\chi, \alpha)$ in the Lie Algebra of the Riordan group, we associate to it the partial differential equation $\frac{\partial u}{\partial t}=\chi(x)u(x,t)+x\al(x)\frac{\partial u}{\partial x}$. At this point, we run into the following {\it dichotomy}.
\begin{itemize}
\item[(a)] If we are able to compute the one-parameter group $e^{tL}$ and to recognize $e^{tL}=T(f(x,t)\mid g(x,t))$ as a Riordan matrix for any $t\in \R$, then this allows us to solve in $\K[[x,t]]$ the initial value problem \[
\left\{
  \begin{array}{ll}
    \frac{\partial u}{\partial t}=\chi(x)u(x,t)+x\al(x)\frac{\partial u}{\partial x} \\
    u(x,0)=h(x)
  \end{array}
\right.
\]
 for any $h\in \K[[x]]$  because the solution is given by \[
u(x,t)=e^{tL(\chi(x),\al(x))}(h(x))=\frac{f(x,t)}{g(x,t)}h\left(\frac{x}{g(x,t)}\right).
\]
\item[(b)] If, on the contrary, we are able to solve in $\K[[x,t]]$ the initial value problem \[
\left\{
  \begin{array}{ll}
    \frac{\partial u}{\partial t}=\chi(x)u(x,t)+x\al(x)\frac{\partial u}{\partial x} \\
    u(x,0)=h(x)
  \end{array}
\right.
\]
 for any $h\in \K[[x]]$, then we get the one-parameter subgroup   $e^{tL(\chi(x),\al(x))}=T(f(x,t)\mid g(x,t))$. We can compute both parameters $f(x,t)$ and $g(x,t)$, because $\frac{f(x,t)}{g(x,t)}$ is the unique solution for the initial condition $h(x)\equiv 1$ and $\frac{f(x,t)}{g(x,t)}\frac{x}{g(x,t)}$ is the unique solution for the initial condition $h(x)=x$. Evaluating now at $t=1$, we have the corresponding matrix exponential.
\end{itemize}

\subsection{From the matrix exponential to the solution of the PDE}
\begin{ejem}
\begin{itemize}
\item[(i)] Consider the matrix

\[
D=\left(
  \begin{array}{ccccc}
    1 & 0 & 0 & 0 & \cdots \\
    0 & 2 & 0 & 0 & \cdots \\
    0 & 0 & 3 & 0 & \cdots \\
    0 & 0 & 0 & 4 & \cdots \\
    \vdots  & \vdots  & \vdots  & \vdots & \ddots \\
  \end{array}
\right).
\]
Obviously, $D=L(1,1)$ is in $\mathcal{L}(\Ri(\K))$. The partial differential equation associated to the matrix $D$ is \[\frac{\partial u}{\partial t}=u(x,t)+x\frac{\partial u}{\partial x}\]
The solution of the corresponding initial value problem \[
\left\{
  \begin{array}{ll}
    \frac{\partial u}{\partial t}=u(x,t)+x\frac{\partial u}{\partial x} \\
    u(x,0)=h(x)
  \end{array}
\right.
\] is given by $u(x,t)=e^{tD}(h)$. It is clear, by definition,  that
\[
e^{tD}=\left(
  \begin{array}{ccccc}
    e^{t} & 0 & 0 & 0 & \cdots \\
    0 & e^{2t} & 0 & 0 & \cdots \\
    0 & 0 & e^{3t} & 0 & \cdots \\
    0 & 0 & 0 & e^{4t} & \cdots \\
    \vdots  & \vdots  & \vdots  & \vdots & \ddots \\
  \end{array}
\right).
\] For any $t\in \R$, we have that $e^{tD}=T(1\mid e^{-t})$ as a Riordan matrix. Therefore, the solution is $u(x,t)=T(1\mid e^{-t})(h)=e^{t}h(xe^{t})$.

\item[(ii)] In this more interesting example, we consider the matrix

\[
H =\left(
  \begin{array}{ccccc}
    0 & 0 & 0 & 0 & \cdots \\
    1 & 0 & 0 & 0 & \cdots \\
    0 & 2 & 0 & 0 & \cdots \\
    0 & 0 & 3 & 0 & \cdots \\
    \vdots  & \vdots  & \vdots  & \vdots & \ddots \\
  \end{array}
\right).
\] Note that $H=L(x,x)$ and so $H\in \mathcal{L}(\Ri(\K))$.  The PDE induced by $H$ is
\[ \frac{\partial u}{\partial t}-x^{2}\frac{\partial u}{\partial x}=xu(x,t) \]

The solution of the corresponding initial value problem
\[
\left\{
  \begin{array}{ll}
    \frac{\partial u}{\partial t}=xu(x,t)+x^{2}\frac{\partial u}{\partial x} \\
    u(x,0)=h(x)
  \end{array}
\right.
\] is given by $u(x,t)=e^{tH}(h)$. To compute this solution we are going to take advantage of some previous work of another authors. Particularly, we are going to follow \cite{Aceto2}, where the matrix $H$ is called the {\it creation matrix}. Formula $(9)$ in \cite[Page 233]{Aceto2} computes $e^{t\Lambda_{n-1}(H)}$. Moreover, it is obvious   that $P_{n-1}(e^{t\Lambda_{n}(H)})=e^{t\Lambda_{n-1}(H)}$ for any $t\in\R$ and any integer $n\geq2$ (see \cite[Page 542]{Lie} for the notation concerning $\Lambda_n$). Hence, in our Riordan matrix notation, we get $e^{tH}=T(1 \mid 1-xt)$ and the solution of the corresponding initial value problem is given by \[u(x,t)=e^{tH}(h)=\frac{1}{1-xt}h\left(\frac{x}{1-xt}\right).\]

Evaluating at $t=1$, we obtain $e^{H}=T(1\mid 1-x)$, which is the {\it Pascal triangle}.

\end{itemize}

\end{ejem}
\subsection{From the solution of a PDE to the matrix exponential: the Method of Characteristics}

 To point out how to go from the solution of a PDE to the exponential map of the Lie algebra to the Riordan group, we are going to deal with the following family of significant examples.

 For any couple of real numbers $a$, $b$ and any non-negative integer number $n$, consider the infinite lower triangular matrix $L^{a,b}_{n}=(d_{i,j}^{(n)})_{i,j\in \mathbb{N}}$, where $d_{i,j}^{(n)}=0$ if $i-j\neq n$ and $d_{i,j}^{(n)}=a +jb$ if $i-j= n$. How can we compute or describe the exponential of each of the matrices $L^{a,b}_{n}$?

Note that $L^{a,b}_{n}=L(ax^n , bx^n)$. This means that they belong to the Lie algebra of the Riordan group. Moreover,

\[
L^{a,b}_{0}=\left(
  \begin{array}{ccccc}
    a & 0 & 0 & 0 & \cdots \\
    0 & a+b & 0 & 0 & \cdots \\
    0 & 0 & a+2b & 0 & \cdots \\
    0 & 0 & 0 & a+3b & \cdots \\
    \vdots  & \vdots  & \vdots  & \vdots & \ddots \\
  \end{array}
\right),
\]

\[
L^{a,b}_1 =\left(
  \begin{array}{ccccc}
    0 & 0 & 0 & 0 & \cdots \\
    a & 0 & 0 & 0 & \cdots \\
    0 & a+b & 0 & 0 & \cdots \\
    0 & 0 & a+2b & 0 & \cdots \\
    \vdots  & \vdots  & \vdots  & \vdots & \ddots \\
  \end{array}
\right),
\] and so on. Observe that  $D=L^{1,1}_0$ and $H=L^{1,1}_1$, where $D$ and $H$ are the matrices considered in the previous examples.
The PDE induced by $L^{a,b}_{n}$ is
\[ \frac{\partial u}{\partial t}-bx^{n+1}\frac{\partial u}{\partial x}=ax^{n}u(x,t) .\]
Note that if $b=0$; then, $e^{tL^{a,0}_{n}}=T(e^{atx^{n}} \mid 1)$ is a Toeplitz matrix (in the Riordan group) for any $t\in\R$. Consequently,  the solution  of the corresponding initial value problem
\[
\left\{
  \begin{array}{ll}
    \frac{\partial u}{\partial t}=ax^{n}u(x,t) \\
    u(x,0)=h(x)
  \end{array}
\right.
\] is given by $u(x,t)=e^{tL^{a,0}_{n}}(h)=e^{atx^{n}}h(x)$.

From now on, assume $b\neq 0$. Consider the matrix $L^{a,b}_{n}$ and the corresponding  initial value problem
\begin{equation}\label{eq:1}
 \quad\left\{
  \begin{array}{ll}
    \frac{\partial u}{\partial t}-bx^{n+1}\frac{\partial u}{\partial x}=ax^{n}u(x,t) \\
    u(x,0)=h(x).
  \end{array}
\right.
\end{equation}
Let us use the Method of Characteristics to solve it (see \cite{john}, Chapter I). Suppose $t=t(r,s)$, $x=x(r,s)$ and $z=z(r,s)$. So we consider the following system of ODEs

\[
\left\{
  \begin{array}{lll}
    \frac{dt}{ds}=1 \\
    \frac{dx}{ds}=-bx^{n+1} \\
    \frac{dz}{ds}=ax^{n}z
  \end{array}
\right.
\]
with initial conditions
\[
\left\{
  \begin{array}{lll}
  t(r,0)=0 \\
  x(r,0)=r \\
  z(r,0)=h(r)
  \end{array}
\right.
\]
Integrating the system and imposing the initial conditions we get
  \[
\left\{
  \begin{array}{lll}
t=s\\
 \frac{1}{x^n}=\frac{1+nbsr^n}{r^n}\\
 z=(1+nbsr^{n})^{\frac{a}{nb}}h(r)
  \end{array}
\right.
\]
and
 \[
\left\{
  \begin{array}{lll}
s=t\\
r=\frac{x}{\sqrt[n]{1-nbtx^{n}}}\\
z=\left(\frac{1}{1-bntx^{n}}\right)^{\frac{a}{nb}}h\left(\frac{x}{\sqrt[n]{1-nbtx^{n}}}\right)
  \end{array}
\right. ,
\]
which yields that the solution of (\ref{eq:1}) is given by \[ u(x,t)=\left(\frac{1}{1-bntx^{n}}\right)^{\frac{a}{nb}}h\left(\frac{x}{\sqrt[n]{1-nbtx^{n}}}\right)\] or
\[ u(x,t)=\sqrt[n]{\left(\frac{1}{1-bntx^{n}}\right)^{\frac{a}{b}}}h\left(\frac{x}{\sqrt[n]{1-nbtx^{n}}}\right).\]
All above in this subsection may be summarized as
\begin{thm} Suppose that  $a$ and  $b$ are two real numbers and that $n$ is a non-negative integer number. Consider the  infinite lower triangular matrix  $L^{a,b}_{n}=(d_{i,j}^{(n)})_{i,j\in \mathbb{N}}$, where $d_{i,j}^{(n)}=0$ if $i-j\neq n$ and $d_{i,j}^{(n)}=a +jb$ if $i-j= n$. Then, $L^{a,b}_{n}$ is in the Lie algebra of the Riordan group and for any $t\in\R$ the Riordan matrix $e^{tL^{a,b}_{n}}$ is given by \[e^{tL^{a,b}_{n}}= \left\{
  \begin{array}{ll}
    T\left(\sqrt[n]{(1-bntx^{n})^{\frac{b-a}{b}}}\mid   \sqrt[n]{1-bntx^{n}}\right)   &     \text { if   $b\neq0$} \\
    T\left(e^{atx^{n}}\mid 1\right) & \text{if b=0}.
  \end{array}
\right.
\]
\end{thm}
\section{A special involution as symmetry or time-reversal symmetry for flows in $\K[[x]]$ related to the Riordan group.}
There is an involution $M$ in the Riordan group playing a special role in such group. This matrix is
\[
M=T(-1|-1)=\left(
  \begin{array}{ccccc}
    1 & 0 & 0 & 0 & \cdots \\
    0 & -1 & 0 & 0 & \cdots \\
    0 & 0 & 1 & 0 & \cdots \\
    0 & 0 & 0 & -1 & \cdots \\
    \vdots  & \vdots  & \vdots  & \vdots & \ddots \\
  \end{array}
\right).
\]

Considered as an endomorphism in $\K[[x]]$, $M$ is continuous and has exactly two eigenvalues $1$ and $-1$. The corresponding eigenspaces are $E(1)=\{h\in \K[x]] /   h(x)=h(-x)\}$ and $E(-1)=\{h\in \K[x]] /  -h(x)=h(-x)\}$, i.e., they are the linear subspaces of even, respectively odd, formal power series. Both of them are closed subspaces in $\K[[x]]$ and we get $\K[[x]]=E(1)\oplus E(-1)$ and $(M-I)\circ (M+I)=(M+I)\circ (M-I)\equiv 0$, where $I=T(1\mid 1)$ is the identity.

The reason of the interest about $M$ is because it is used for defining what is known as a  pseudo-involution in the Riordan group. Following \cite{Nkwanta1}, we say that a Riordan matrix $R$ is a {\it pseudo-involution} if the product $RM$ is an involution, that is, $RMRM=I$.

In Group Theory we have the definitions of {\it reversible} and {\it strongly reversible} elements (see \cite{ofarrel2}). We recall it here for completeness.
\begin{defi}
\begin{itemize}
\item[(i)] An element $g$ of a group $G$ is said to be reversible in $G$ if there is another element $h$ of $G$ such that \[hgh^{-1} = g^{-1}.\] In this situation we also say that $h$ reverses $g$ and $h$ is a reverser of $g$.
    \item[(ii)] An element $g$ of a group $G$ is said to be strongly reversible in $G$ if there is an involution  $h$ of $G$ such that \[hgh^{-1} = g^{-1}.\]
\end{itemize}
\end{defi}

\begin{prop}If $R$ is a pseudo-involution, then $R$ is strongly reversible in the Riordan group and is the product of two involutions.
\end{prop}
\begin{proof}
Since $RMRM=I$ we directly obtain, multiplying  on the left by the inverse $R^{-1}$, that $MRM=R^{-1}$ so $R$ is strongly reversible. Moreover, $R^{-1}M$ is an involution because $(R^{-1}M)^{-1}=MR=MMR^{-1}M=R^{-1}M$. Consequently, we have that $R=MR^{-1}M$. 
\end{proof}

The above proposition  tells us that the pseudo-involutions are particular examples of {\it strongly reversible} elements in the Riordan group and $M$ is a reverser for any of them.

Consider the left and right translations in $\Ri(\K)$ given,
respectively, by
\[
\begin{matrix}
L_D& : & \Ri(\K) &  \longrightarrow & \Ri(\K)\\
&& X & \longmapsto & L_D(X)=DX,
\end{matrix}\qquad \qquad \begin{matrix}
R_D& : & \Ri(\K) &  \longrightarrow & \Ri(\K)\\
&& X & \longmapsto & R_D(X)=XD.
\end{matrix}
\]
Since the product is a $C^{\infty}$-function  both $L_D$ and $R_D$
are diffeomorphisms. We need to recall the following facts from \cite{Lie}.
\begin{prop}\label{T: conj}
Let $T(f\mid g)$ be a Riordan matrix. The tangent space $T_{T(f\mid
g)}\Ri(\K)$ to $\Ri(\K)$ at $T(f\mid g)$ is given by
\[
T_{T(f\mid g)}\Ri(\K)=\{T(f\mid g)L(\chi,\al)| \chi,
\al\in\K[[x]]\}= \{L(\chi,\al)T(f\mid g)| \chi, \al\in\K[[x]]\}.
\]
Moreover, the conjugation by $T(f\mid g)$ defined by
\[
\begin{matrix}
\text{conj}_{T(f\mid g)}& : & \Ri(\K) &  \longrightarrow & \Ri(\K)\\
&& X & \longmapsto & T(f\mid g)XT^{-1}(f\mid g)
\end{matrix}
\]
is a $C^{\infty}$-diffeomorphism and its tangent (or differential) map at the identity
\[
\begin{matrix}
D\text{conj}_{T(f\mid g)}(I)& : & \mathcal{L}(\Ri(\K)) &
\longrightarrow & \mathcal{L}(\Ri(\K))
\end{matrix}
\]
is given by
\[
D\text{conj}_{T(f\mid g)}(I)(L(\chi,\al))=T(f\mid
g)L(\chi,\al)T^{-1}(f\mid g).
\]
Finally, for any $t\in\K$ we have \large{
\[
e^{tD\text{conj}_{T(f\mid g)}(I)(L(\chi,\al))}=\text{conj}_{T(f\mid
g)}(e^{tL(\chi,\al)}).
\]}
\end{prop}
From this, we deduce
\begin{cor}
Given $T(f\mid g)\in\Ri(\K)$ and
$L(\chi,\al)\in\mathcal{L}(\Ri(\K))$, there exists a unique
$L(\tilde{\chi},\tilde{\al})\in\mathcal{L}(\Ri(\K))$ such that
\[
T(f\mid g)L(\chi,\al)=L(\tilde{\chi},\tilde{\al})T(f\mid g)
\]
or equivalently
\[
D\text{conj}_{T(f\mid
g)}(I)(L(\chi,\al))=L(\tilde{\chi},\tilde{\al}).
\]
Moreover,
\begin{equation}\label{eq:chialpha}
\tilde{\chi}=\frac{\chi\left(\frac{x}{g}\right)(g-xg')f-x\al\left(\frac{x}{g}\right)(f'g-g'f)}{f(g-xg')},
\qquad \tilde{\al}=\frac{g\al\left(\frac{x}{g}\right)}{g-xg'}.
\end{equation}
\end{cor}

As a consequence, we obtain
\begin{prop}\label{T: eigen} Consider the involution $M$ and the $C^{\infty}$-diffeomorphism \[\begin{matrix}
\text{conj}_{M}& : & \Ri(\K) &  \longrightarrow & \Ri(\K)\\
&& X & \longmapsto & MXM.
\end{matrix}
\]
Then
\begin{itemize}
\item[(i)] the tangent (or differential) map at identity
\[
\begin{matrix}
D\text{conj}_{M}(I)& : & \mathcal{L}(\Ri(\K)) &
\longrightarrow & \mathcal{L}(\Ri(\K))
\end{matrix}
\]
is a linear involution
\item[(ii)] for any couple  of real numbers $a$ and  $b$ such that $a\neq 0$ and $b\neq 0$ the  matrix $L^{a,b}_{n}$ is an eigenvector of $D\text{conj}_{M}(I)$ corresponding to the eigenvalue $-1$ if $n$ is odd and an eigenvector of $D\text{conj}_{M}(I)$ corresponding to the eigenvalue $1$ if $n$ is even (including $n=0$).
\end{itemize}
\end{prop}
\begin{proof} It is clear that $D\text{conj}_{M}(I)\circ D\text{conj}_{M}(I)=I_{\mathcal{L}(\Ri(\K))}$ because $M$ is an involution in $\Ri(\K)$, where $I_{\mathcal{L}(\Ri(\K))}$ represents the identity map in $\mathcal{L}(\Ri(\K)) .$ To prove (ii), recall that $M=T(-1\mid -1)$ and that $L^{a,b}_{n}=L(ax^n , bx^n)$. From the previous corollary we get \[
D\text{conj}_{M}(I)(L(\chi, \al))=L(\tilde{\chi},\tilde{\al})\] where 
$\tilde{\chi}$ and $\tilde{\al}$ are given by (\ref{eq:chialpha}). In this case, $f=g=-1$ and $\chi(x)=ax^{n}$, $\al(x)=bx^{n}.$ For the case that $n$ even, the announced result is obvious. The same also holds for the case $n$ is odd because $ -L(\chi,\al)=L(-\chi,-\al)$ for any $\chi$, $\al$ in $\K[[x]].$
\end{proof}

Finally, we obtain the following (here we use \cite{Lam-Rob} for some of the definitions appearing below). 
\begin{thm}\label{T: TR} Let  $a$ and $b$ two real numbers such that $a=b=0$ does not hold. Suppose also that  $n$ is a non-negative integer number. Consider the infinite lower triangular matrix $L^{a,b}_{n}$. Then we have the following. 
\begin{itemize}
\item[(i)] If $n$ is even, the one-parameter subgroup $\{e^{tL^{a,b}_{n}}\}_{t\in \R}$ is contained in the centralizer of the involution $M$. In other words, $M$ is a symmetry for the flow  $\{e^{tL^{a,b}_{n}}\}_{t\in \R}$ in $\K[[x]]$.
    \item[(ii)] If $n$ is odd, any element in the one-parameter subgroup $\{e^{tL^{a,b}_{n}}\}_{t\in \R}$ is a pseudo-involution and the involution $M$ is a time-reversal symmetry of the flow $\{e^{tL^{a,b}_{n}}\}_{t\in \R}$ in $\K[[x]]$.
\end{itemize}
\end{thm}
\begin{proof} Suppose $n$ is even. Using Proposition \ref{T: eigen}, we have $D\text{conj}_{M}(I)(L^{a,b}_{n})=L^{a,b}_{n}$. Hence \[e^{tL^{a,b}_{n}}=e^{tD\text{conj}_{M}(I)(L^{a,b}_{n})}=Me^{tL^{a,b}_{n}}M.\] The last equality above is a consequence of Proposition \ref{T: conj}. Consequently, $e^{tL^{a,b}_{n}}$ is in the centralizer of the involution $M$ for every $ t\in \R$.

Suppose now that $n$ is odd. By analogous arguments, we first have that $D\text{conj}_{M}(I)(L^{a,b}_{n})=-L^{a,b}_{n}$ and then
\[e^{-tL^{a,b}_{n}}=e^{tD\text{conj}_{M}(I)(L^{a,b}_{n})}=Me^{tL^{a,b}_{n}}M.\] Therefore, the Riordan matrix $e^{tL^{a,b}_{n}}$ is reversible for any $ t\in \R $ and the involution $M$ is a reverser for all of them. This fact implies that $e^{tL^{a,b}_{n}}$ is strongly reversible and that  $e^{tL^{a,b}_{n}}$ is a pseudo-involution because $e^{tL^{a,b}_{n}}M$ is an involution ($e^{tL^{a,b}_{n}}Me^{tL^{a,b}_{n}}M=e^{tL^{a,b}_{n}}e^{-tL^{a,b}_{n}}=I)$.
\end{proof}
\begin{cor}\label{T: CTR}The same as in the above theorem is true, word by word, changing the involution $M$ by the involution $-M=T(1 \mid -1).$
\end{cor}

\begin{section}{On some  dynamical properties  of Pascal Triangle}

Using some of the results obtained herein we can get some new information about the Riordan group. For example, in \cite{Lie}  we recognized as a subgroup of the Riordan group, the {\it substitution group of formal power series}. This  group was introduced in \cite{Jenn} (see also \cite{Babenko} for a good survey about it) where  results about topological generation are established. One can notice at once that some of our one-parameter groups $\{e^{tL^{a,b}_{n}}\}_{t \in \R}$ are involved in \cite[Section 3]{Jenn}. In this section, we focus only on the Pascal Triangle and  assign to it others of the many properties related to patterns and symmetries that it has.

We will use the pro-Lie group structure of the Riordan group. So, we will consider elements of the Riordan group and of the Lie algebra as approximated by the components  of the corresponding points in the inverse limit interpretation of both $\Ri(\K)$ and $\mathcal{L}(\Ri(\K))$ (see \cite{Lie} for the pro-structure of $\mathcal{L}(\Ri(\K))$).  

To clarify the above approaching process recall that related to any problem \begin{equation}\label{E_: GEN}
 \gamma'(t)=L(\gamma(t))
\end{equation} with $L\in\mathcal{L}(\Ri(\K))$ or, equivalently \[\frac{\partial u}{\partial t}=\chi(x)u(x,t)+x\al(x)\frac{\partial u}{\partial x},\] if $L=L(\chi(x), \al(x))$,  we have a sequence of finite dimensional problems, denoted by $\{(\ref{E_: GEN})_{n}\}_{n\in \N}$. We call this sequence as the sequence of {\it approaching}  problems of  $(\ref{E_: GEN})$.

{\bf Problem $(\ref{E_: GEN})_{n}$: Approaching problems.}
\textit{
Consider the Euclidean space $\R^{n+1}$. For any $x\in \R^{n+1}$ denote by $x=(x_{0}, x_{1}, \cdots, x_{n})$ to its usual components. Suppose $x^{T}$ represents the transpose matrix of $x.$ Let $x:\R\longrightarrow \R^{n+1}$ be any derivable curve given by $x(t)=(x_{0}(t), x_{1}(t), \cdots, x_{n}(t)).$ We denote by $x'(t)=(x'_{0}(t), x'_{1}(t), \cdots, x'_{n}(t))$ the derivative of $x$  at $t$, where $x'_{i}$ is the usual derivative of a real function with real variable for any $i=0,\cdots, n$. When we refer to the {\it approaching problem} $(\ref{E_: GEN})_{n}$,  we refer to the linear differential equation \[x'^{T}= D\Pi_{n}(I)(L)x^{T}.\] In the above equation $D\Pi_{n}(I)$ represents the differential at the identity $I$ in the Riordan group $\Ri(\K)$ of the projection \[\Pi_{n}:\Ri(\K)\longrightarrow \Ri_{n}(\K)\] in the pro-Lie group structure of $\Ri(\K)$, see again \cite{Lie} if needed.}

\subsection{The approaching problems related to Pascal Triangle}

Recall that, Pascal Triangle, is the time $1$ map of the dynamical system generated by the linear differential equation in $\K[[x]]$
\begin{equation}\label{E_:ECU}
 \gamma'(t)=L^{1,1}_1(\gamma(t)),
\end{equation} where $L^{1,1}_1=L(x,x)$, or, equivalently, \[ \frac{\partial u}{\partial t}=xu + x^{2}\frac{\partial u}{\partial x} \] in $\K[[x,t]].$
Recall also that $\gamma(t)\in \R[[x]]$ for any $t\in \R$ and that $L^{1,1}_1(\gamma(t))=x\gamma(t)+x^{2}\frac{d\gamma(t)}{dx}$, where
$\frac{d}{dx}$ is the formal derivative in $\R[[x]]$. While $\gamma:\R \longrightarrow\R[[x]]$ is a curve and $\gamma'(t)$ is the usual derivative in $t$, when we consider $\R[[x]]$ identified with $\R^\N$ with the product topology. Using the pro-Lie group structure in $\Ri(\K)$, the corresponding pro-Lie algebra structure in $\mathcal{L}(\Ri(\K))$ and recalling that \[
L^{1,1}_1 =\left(
  \begin{array}{ccccc}
    0 & 0 & 0 & 0 & \cdots \\
    1 & 0 & 0 & 0 & \cdots \\
    0 & 2 & 0 & 0 & \cdots \\
    0 & 0 & 3 & 0 & \cdots \\
    \vdots  & \vdots  & \vdots  & \vdots & \ddots \\
  \end{array}
\right).
\]

We can associate to problem $(\ref{E_:ECU})$ a countably infinite  family of finite dimensional problems $(\ref{E_:ECU})_{n}$ in the euclidean space $\R^{n+1}$ for any non-negative integer $n$.  As we said before, we  will interpret the family of problems $\{(\ref{E_:ECU})_{n}\}_{n\in \N}$ as {\it approaching} the problem $(\ref{E_:ECU})$ when $n$ tends to $\infty$. For the first few values of $n$ we have:
\begin{itemize}
\item[$(\ref{E_:ECU})_{0}$] for $n=0$, we consider the differential equation $x'_{0}=0$ in the one dimensional euclidean space $\R$.
\item[$(\ref{E_:ECU})_{1}$] for $n=1$, we consider the differential equation in $\R^{2}$
\[\begin{pmatrix}x'_{0} (t) \\ x'_{1} (t) \end{pmatrix} = \begin{pmatrix}0&0\\  1& 0
\end{pmatrix}\begin{pmatrix}x_{0} (t) \\ x_{1} (t) \end{pmatrix} \]
\item[$(\ref{E_:ECU})_{2}$] for $n=2$, we consider the differential equation in $\R^{3}$
\[\begin{pmatrix}x'_{0} (t) \\ x'_{1} (t) \\x'_2 (t) \end{pmatrix} = \begin{pmatrix}0&0&0\\  1 & 0 & 0 \\ 0 & 2 & 0
\end{pmatrix}\begin{pmatrix}x_{0} (t) \\ x_{1} (t) \\x_{2}(t) \end{pmatrix} \]
\item[$(\ref{E_:ECU})_{3}$] for $n=3$, we consider the differential equation in $\R^{4}$
\[\begin{pmatrix}x'_{0} (t) \\ x'_{1} (t) \\x'_2 (t) \\x'_{3}(t) \end{pmatrix} = \begin{pmatrix}0&0&0&0\\  1 & 0 & 0 & 0\\ 0 & 2 & 0 & 0 \\ 0 & 0 & 3 & 0
\end{pmatrix}\begin{pmatrix}x_{0} (t) \\ x_{1} (t) \\x_{2}(t) \\ x_{3} (t) \end{pmatrix} \]

\end{itemize}
and so on.

The problem $(\ref{E_:ECU})_{0}$ is very easy to analyze. Any point in the phase space $\R$ is an equilibrium point. The phase flow is trivial and nothing is moving under it. 

Let us now consider $(\ref{E_:ECU})_{1}$. The equilibrium points in this case are just the points in the $x_{1}$-axe, which means that they are of the form $(0,b)$, where $b\in \R$. All of them are unstable in the Liapunov sense. The rest of the orbits are the straight lines $x_{0}=a$ for a fix non-null $a\in \R.$ Through any orbit the motion is uniform. In the semiplane $x_{0}>0$ the sense of the motion is increasing, respect to the $x_{1}$-axe as $t$ increases, i.e., the particle comes from the $-\infty$ part respect to the $x_{1}$-axe and goes to (positive) $\infty$ of such axe as $t$ increases. In the semiplane $x_{0}<0$ the sense of the motion is the opposite one. The constant speed of the motion in the orbit $x_{0}=a$  is $\mid a \mid$, the absolute value of $a.$ Particles move quickly for large values of $\mid a \mid$ and slowly for small values $\mid a \mid$ and they do not move in the $x_{1}$-axe. Therefore, there seems to be something in the $x_{1}$-axe slowing down the motion. We can deduce all above only knowing that the corresponding velocity vector field for the  equation $(\ref{E_:ECU})_{1}$, in the euclidean plane,  is given by $X(a,b)=(0,a)$. We can also compute, quickly and easily the solution of problem  $(\ref{E_:ECU})_{1}$ with initial condition  $x_{0}=a$ and $x_{1}=b$ because the corresponding matrix is nilpotent. It is the curve $x(t)=(a, at+b)$.

From  Theorem $\ref{T: TR}$ and Corollary $\ref{T: CTR}$, we get that $\Pi_1 (M)=\begin{pmatrix}1&0\\  0& -1
\end{pmatrix}$  and $\Pi_1(-M)=\begin{pmatrix}-1&0\\  0& 1
\end{pmatrix}$ are time-reversal symmetries for the flow generated by the equation $(\ref{E_:ECU})_{1}$. Consider the orbit $\theta(a,b)=\{(a,at+b), a\neq0, b, t \in \R \}$ of $(\ref{E_:ECU})_{1}$. Then it is symmetric respect to the involution $\Pi_1 (M)$, see Definition 4.1 in \cite{Lam-Rob}. This, in particular, means that if we transform the orbit $\theta(a,b)$ by  $\Pi_1 (M)$ we get again $\theta(a,b)$ but the parametrization obtained by means of $t$ is not a solution. However, if we finally change $t$ by $-t$ in such obtained curve we have a solution of $(\ref{E_:ECU})_{1}$; in this case the initial value at $t=0$ is $(a,-b)$. Of course, $\theta(a,b)=\theta(a,-b)$.

What is the behaviour under the action of the involution $\Pi_1(-M)$? if we transform the orbit $\theta(a,b)$ by means of $\Pi_1(-M)$, we get another {\it different } orbit of $(\ref{E_:ECU})_{1}$. In fact, we obtain  $\Pi_1(-M)(\theta(a,b))=\theta(-a,b)$ and $\theta(a,b)\neq\theta(-a,b). $ Again, the parametrization so obtained by means of $t$ is not a solution of $(\ref{E_:ECU})_{1}$. But, again, if we change $t$ by $-t$ we get another solution but with {\it different} orbit. Finally see that the composition of both time reversal symmetries is really a {\it symmetry} for $(\ref{E_:ECU})_{1}$. Then we obtain that $-I$ is a symmetry and it implies that if  $x(t)=(a, at+b)$ is a solution of the problem with initial condition $x(0)=(a, b)$, then $-x(t)=-I(x(t))$ is a solution with initial condition $(-a,-b).$

\begin{rmk} Note that any orbit $x_{0}=a\neq0$ of $(\ref{E_:ECU})_{1}$  is symmetric respect to the time reversal symmetry $\Pi_1(M)$. On the contrary, no orbit (except for equilibrium points) is symmetric respect to the time reversal symmetry $\Pi_1(-M).$
\end{rmk}

We propose to the reader the beautiful exercise of analysing the problem $(\ref{E_:ECU})_{2}$ and the role of the involutions $\Pi_2 (M)=\begin{pmatrix}1&0&0\\  0& -1&0 \\ 0&0&1
\end{pmatrix}$ and $\Pi_2 (-M)=\begin{pmatrix}-1&0&0\\  0& 1&0 \\ 0&0&-1
\end{pmatrix}$ as time reversal symmetries for $(\ref{E_:ECU})_{2}$. The geometric analysis of this problem is, obviously, richer than that of $(\ref{E_:ECU})_{1}$. In the problem $(\ref{E_:ECU})_{1}$ the involutions $\Pi_1(M)$ and $\Pi_1(-M)$ are conjugated and they represent reflections about different axes. But in $(\ref{E_:ECU})_{2}$, $\Pi_2 (M)$ is a reflection about the plane $x_{1}=0$, which is  a plane of fixed points of $\Pi_2 (M)$, and $\Pi_2 (-M)$ is a rotation of angle $\pi$ around the  $x_{1}$-axe. Of course they are not conjugated. Analysing this case,  one can see first that the equilibrium points are those of the form $(0,0,c)$ and that under the flow generated by $(\ref{E_:ECU})_{2}$, particles are moving within affine hyperplanes $x_{0}=a$. When $a=0$ we obtain the motion described in $(\ref{E_:ECU})_{1}$ in this hyperplane. For $a\neq0$ the orbits of points are parabolas. In this case, the velocity vector field is given by $X(a,b,c)=(0,a,2b)$ and the solution of $(\ref{E_:ECU})_{2}$ with initial conditions $x_{0}(0)=a$, $x_{1}(0)=b$ and $x_{2}(0)=c$ is given by $x(t)=(a, at+b, at^{2}+2bt+c)$. As in the previous case, any orbit, which is not an equilibrium point, is symmetric with respect to the involution $ \Pi_2 (M)$, but any non-trivial orbit is transformed by    $\Pi_2 (-M)$ into another different orbit. In both cases, if after the transformation, we change $t$ by $-t$ we obtain new solutions of the problem $(\ref{E_:ECU})_{2}.$ Finally $-I:\R^{3}\longrightarrow \R^{3}$ is a symmetry for the problem.

\subsection{Facts and/or conjectures and/or speculations on the problems $(\ref{E_:ECU})_{n}$ and $(\ref{E_:ECU})$ }

The flow induced by problem $(\ref{E_:ECU})$ in $\K[[x]]$ is given by \[\Phi:\K[[x]]\times \R \longrightarrow \K[[x]] \] where $\Phi(h,t)=e^{tL^{1,1}_{1}}(h)= \frac{1}{1-xt}h\left(\frac{x}{1-xt}\right)$, while the flow generated by the problem $(\ref{E_:ECU})_{n}$ in $\R^{n+1}$ is \[\Phi_{n}:\R^{n+1}\times \R \longrightarrow \R^{n+1} \] whose matrix expression is $\Phi_{n}(x,t)=\Pi_{n}(e^{tL^{1,1}_{1}})x^{T}.$

We now are going to state, without proofs, properties related to problems $(\ref{E_:ECU})_{n}$ and $(\ref{E_:ECU})$. This is the reason why we entitled this subsection as we did.
\begin{prop}{\bf (Dynamical properties related to problems \textbf{$(\ref{E_:ECU})_{n})$}} 
Let $n$ be a non-negative integer number, then we have the following properties.
\begin{itemize}
\item[(i)] The orbit of any point $a=(a_{0}, a_{1}, \cdots, a_{n})$ in $\R^{n+1}$ is contained in the affine hyperplane $x_{0}=a_{0}.$ Moreover if $a_{0}=0$ and if one consider $\R^{n}=\{(x_{0}, x_{1}, \cdots, x_{n})\in \R^{n+1} / x_{0}=0\}$, the motion induced by $\Phi_{n}$ in $\R^{n}$ is $\Phi_{n-1}$.
    \item[(ii)] The equilibrium points in $(\ref{E_:ECU})_{n}$ are those in the $x_{n}$-axis and if $n>0$ all of them are unstable in the Liapunov sense.
    \item[(iii)] The non-trivial orbits in $(\ref{E_:ECU})_{n}$, i.e., those which are not equilibrium points, are related to the so called {\it moment curve} in the corresponding hyperplane. In particular the solution of $(\ref{E_:ECU})_{n}$  with initial condition $(1, 0, \cdots, 0)\in \R^{n+1}$ is $x(t)=(1, t,t^{2} \cdots, t^{n})$ which is a copy of the corresponding moment curve in the hyperplane $x_{0}=1.$
    \item[(iv)] Any non-trivial orbit in $(\ref{E_:ECU})_{n}$ is symmetric with respect to the time-reversal symmetry $\Pi_{n}(M)$ and no one of them is symmetric respect to $\Pi_{n}(-M).$ Anyway, if we have a non-trivial solution of $(\ref{E_:ECU})_{n}$, i.e., a non-constant one, and we transform it by any of the involutions $\Pi_{n}(M)$ or $\Pi_{n}(-M)$ and then change $t$ by $-t$ we get another solution of $(\ref{E_:ECU})_{n}.$ 
    \item[(v)] $-I:\R^{n+1}\longrightarrow \R^{n+1}$ is a symmetry for the equation $(\ref{E_:ECU})_{n}.$
\end{itemize}
   
\end{prop}

Motivated by the previous result and knowing that Pascal triangle is the time one map of the flow   $\Phi$, we can state:

\begin{prop}{\bf (Some dynamical properties of Pascal Triangle)}
\begin{itemize}
\item[(i)] The flow $\Phi$ has not equilibrium points up the null power series, that is,  all the coefficients of the power series are null.
\item [(ii)] The orbit of the formal power series constantly $1$, by means of $\Phi$, is the set of geometric progressions $\{\frac{1}{1-tx}\}_{t\in \R}. $ Then one can think about $1$ moving through the set of the germs of analytic functions, at $x=0$, because for any $t\in \R$ the function $f_{t}(x)=\frac{1}{1-tx}$ for  $x\in (-\frac{1}{\mid t \mid}, \frac{1}{\mid t \mid})$  is analytic at $x=0.$ Moreover, this orbit can be seen as the asymptotic behaviour, as $n$ goes to $\infty$, of the moment curves.
\item[(iii)] The Riordan involution $M$ and $-M$ are time-reversal symmetries for the flow $\Phi$. Any orbit of problem $(\ref{E_:ECU})$ is symmetric respect to $M$ and no orbit, except for the unique equilibrium point, is symmetric respect to $-M$. Anyway, if we transform any solution of $(\ref{E_:ECU})$ by means of $M$ or $-M$ and then change $t$ by $-t$ we get another solution of the problem.
\item[(iv)] $-I:\K[[x]]\longrightarrow \K[[x]]$ is a symmetry for the equation $(\ref{E_:ECU}).$
\end{itemize}
\end{prop}

To finish this paper note the following.

{\bf Claim:} {\it  Any non-trivial orbit of the problem  $(\ref{E_:ECU})_{n}$ has empty $\alpha$-limit set and $\omega$-limit set.} On the other hand, the problem $(\ref{E_:ECU})_{n}$ has time reversal symmetries. This fact allows us to think that there should be relationships between both limit sets. We then decided to force the corresponding flows, by means of considering a compactification of  the corresponding phase spaces, to get non-empty $\alpha$-limit and  $\omega$-limit sets and to look for relationships between them. We proceed as follows.

Consider the one point (or Alexandroff) compactification of $\R^{n+1}$ which is, topologically, the $n+1$-dimensional sphere $S^{n+1}.$ Let us denote by $\infty$ the added point. Note that any $t$-map of the flow $\Phi_{n}$, $\Phi^{t}_{n}$, can be continuously extended to a map $\widetilde{\Phi^{t}_{n}}:S^{n+1}\longrightarrow S^{n+1}$ defining only $\widetilde{\Phi^{t}_{n}}(\infty)=\infty.$ In this way we get a dynamical system
\[\widetilde{\Phi_{n}}:S^{n+1}\times \R \longrightarrow S^{n+1}. \]
For every non-negative integer $n$ we can also extend the time reversal symmetries  $\Pi_{n}(M)$ and $\Pi_{n}(-M)$ to continuous maps
$\widetilde{\Pi_{n}(M)}$  and $\widetilde{ \Pi_{n}(-M)}$  from $S^{n+1}$ onto itself imposing that the point $\infty$ is a fixed point for both of them. We can identify, topologically,   $\R^{n+1}$   with $S^{n+1}\backslash \{\infty\}$ (which is a dense subset of $S^{n+1}$). With all these constructions we have
\begin{prop} Let $n$ be a non-negative integer number. Then we have the following.
\begin{itemize}
\item[(i)] The maps $\widetilde{\Pi_{n}(M)}$ and $\widetilde{\Pi_{n}(-M)}$  are continuous involutions in  $S^{n+1}$ and they are time-reversal symmetries for the dynamical system $\widetilde{\Phi_{n}}.$
\item[(ii)] The orbits of the dynamical system $\widetilde{\Phi_{n}}$ are those of $\Phi_{n}$ (after the mentioned identification) plus  $\{\infty\}$ which is an equilibrium point. Moreover, every non-trivial orbit of $\widetilde{\Phi_{n}}$ is {\it homoclinic}, being the point $\infty$ an attractor and a repeller of all of them.

\end{itemize}
Consequently, the $\alpha$-limit  and the $\omega$-limit  sets of any non-trivial orbit coincide.
\end{prop}

\end{section}

\end{document}